\documentclass[11pt]{amsart}
\usepackage[left=2.7cm,right=2.7cm,top=3.5cm,bottom=3cm]{geometry}

\usepackage{helvet, color}

\usepackage{amscd,amsmath,amsxtra,amsthm,amssymb,stmaryrd,xr,mathrsfs,mathtools,enumerate,enumitem}

\usepackage[all,cmtip]{xy}

 \DeclareFontFamily{U}{wncy}{}
    \DeclareFontShape{U}{wncy}{m}{n}{<->wncyr10}{}
    \DeclareSymbolFont{mcy}{U}{wncy}{m}{n}
    \DeclareMathSymbol{\Sh}{\mathord}{mcy}{"58}

\usepackage{comment}

\definecolor{Green}{rgb}{0.0, 0.5, 0.0}

\theoremstyle{plain}
\newtheorem{theorem}{Theorem}[section]
\newtheorem{lemma}[theorem]{Lemma}

\newtheorem{proposition}[theorem]{Proposition}
\newtheorem{corollary}[theorem]{Corollary}

\numberwithin{equation}{section}

\theoremstyle{definition}
\newtheorem{definition}[theorem]{Definition}

\theoremstyle{remark}
\newtheorem{remark}[theorem]{Remark}

\newcommand{\Tr}{\operatorname{Tr}}
\newcommand{\Gal}{\operatorname{Gal}}

\newcommand{\Qp}{\mathbb{Q}_p}
\newcommand{\Zp}{\mathbb{Z}_p}
\newcommand{\ZZ}{\mathbb{Z}}
\newcommand{\HH}{\mathbb{H}}

\newcommand{\cH}{\mathcal{H}}

\newcommand{\Z}{\mathbb{Z}}
\newcommand{\cG}{\mathcal{G}}

\newcommand{\cE}{\mathcal{E}}
\newcommand{\fE}{\mathfrak{E}}
\newcommand{\p}{\mathfrak{p}}
\newcommand{\Q}{\mathbb{Q}}
\newcommand{\F}{\mathbb{F}}
\newcommand{\Fp}{\F_p}
\newcommand{\fM}{\mathfrak{M}}

\newcommand{\cO}{\mathcal{O}}

\newcommand{\image}{\mathrm{im}}

\newcommand{\Hom}{\mathrm{Hom}}
\newcommand{\Sel}{\mathrm{Sel}}

\newcommand{\rank}{\mathrm{rank}}

\newcommand{\N}{\mathbb{N}}

\newcommand{\TT}{\mathbb{T}}

\newcommand{\tor}{\mathrm{tor}}

\newcommand{\res}{\mathrm{res}}

\newcommand{\cor}{\mathrm{cores}}
\newcommand{\tr}{\mathrm{tr}}

\newcommand{\longmono}{\mbox{\;$\lhook\joinrel\longrightarrow$\;}}

  \DeclareFontFamily{U}{wncy}{}
  \DeclareFontShape{U}{wncy}{m}{n}{<->wncyr10}{}
  \DeclareSymbolFont{mcy}{U}{wncy}{m}{n}
  \DeclareMathSymbol{\sha}{\mathord}{mcy}{"58}

\begin{document}

\title[$\Lambda$-submodules of finite index of plus and minus Selmer groups]{$\Lambda$-submodules of finite index of anticyclotomic\\ plus and minus Selmer groups of elliptic curves}

\author{Jeffrey Hatley, Antonio Lei and Stefano Vigni}

%\author[J.~Hatley]{Jeffrey Hatley}
\address[Hatley]{
Department of Mathematics\\
Union College\\
Bailey Hall 202\\
Schenectady, NY 12308\\
USA}
\email{hatleyj@union.edu}

%\author[A.~Lei]{Antonio Lei}
\address[Lei]{D\'epartement de math\'ematiques et de statistique\\
Pavillon Alexandre-Vachon\\
Universit\'e Laval\\
Qu\'ebec, QC, Canada G1V 0A6}
\email{antonio.lei@mat.ulaval.ca}

%\author[S.~Vigni]{Stefano Vigni}
\address[Vigni]{Dipartimento di Matematica\\
Universit\`a di Genova\\
Via Dodecaneso 35\\
16146 Genova \\
Italy}
\email{vigni@dima.unige.it}

\begin{abstract}
Let $p$ be an odd prime and $K$ an imaginary quadratic field where $p$ splits. Under appropriate hypotheses, Bertolini showed that the Selmer group of a $p$-ordinary elliptic curve over the anticyclotomic $\Z_p$-extension of $K$ does not admit any proper $\Lambda$-submodule of finite index, where $\Lambda$ is a suitable Iwasawa algebra. We generalize this result to the plus and minus Selmer groups (in the sense of Kobayashi) of $p$-supersingular elliptic curves. In particular, in our setting the plus/minus Selmer groups have $\Lambda$-corank one, so they are not $\Lambda$-cotorsion. As an application of our main theorem, we prove results in the vein of Greenberg--Vatsal on Iwasawa invariants of $p$-congruent elliptic curves, extending to the supersingular case results for $p$-ordinary elliptic curves due to Hatley--Lei.
\end{abstract}

\thanks{The second named author's research is supported by the NSERC Discovery Grants Program RGPIN-2020-04259 and RGPAS-2020-00096.}
\thanks{The third named author's research is partially supported by PRIN 2017 ``Geometric, algebraic and analytic methods in arithmetic''.}


\maketitle

\section{Introduction}\label{section:intro}

When studying Selmer groups in the context of Iwasawa theory, it is often desirable to show that these Selmer groups have no proper $\Lambda$-submodules of finite index or, equivalently, that their Pontryagin duals have no nontrivial finite $\Lambda$-submodules, where $\Lambda$ is an appropriate $p$-adic Iwasawa algebra for a prime number $p$ (to fix ideas, in this introduction we can take $\Lambda$ to be the $\Z_p$-algebra $\Z_p[\![X]\!]$ of formal power series over the $p$-adic integers $\Z_p$). For instance, a finitely generated $\Lambda$-module $M$ (which, in our context, will always be the Pontryagin dual of a suitable Selmer group) admits a map 
\begin{equation*}
M \longrightarrow \Lambda^{\oplus r}\oplus \bigoplus_{i=1}^s \Lambda/ (p^{a_i})\oplus \bigoplus_{j=1}^t\Lambda\big/\bigl(F_j^{n_j}\bigr)
\end{equation*}
with finite kernel and cokernel, for integers $r,s,t\geq0$, $a_i,n_j\geq1$ and irreducible Weierstrass polynomials $F_j\in\Z_p[X]$. If $M$ has no nontrivial finite $\Lambda$-submodules, then this map is in fact injective. 

Such non-existence theorems have played a crucial role in certain strategies for proving cases of the Iwasawa Main Conjecture, as in \cite{epw}, \cite{gv}, \cite{kim}, and by now they have been studied in great generality (see, e.g., \cite{greenberg-structure}), although many specific cases of interest remain open to inquiry.

In \cite{Ber1} and \cite{Ber2}, Bertolini approached this question for the Selmer groups attached to rational elliptic curves with good ordinary reduction at a prime $p$ along  the anticylotomic $\Z_p$-extension of a suitable imaginary quadratic field $K$, which is the unique $\Z_p$-extension of $K$ that is Galois and non-abelian over $\Q$. 

The splitting behavior in $K$ of $p$ and of the prime factors of the conductor $N$ of $E$ exerts a tremendous influence on the structure of the Selmer groups that one can attach to $E$ and $K_\infty$. In particular, when $N$ is divisible by an odd number of primes that are inert in $K$ the situation is remarkably similar to the cyclotomic (over $\Q$ or over $K$) setting. When the number of such primes is even, however, the sign of the functional equation for the relevant $L$-function is $-1$, which forces the Selmer groups to grow and thus become more complicated (see, e.g., \cite{Ber1} and \cite{LV-BUMI}).

Another situation in which the Iwawawa theory for elliptic curves is more delicate is when $E$ has good supersingular reduction at $p$. In this case, even in the cyclotomic setting over $\Q$ the usual Selmer groups fail to satisfy a control theorem as in \cite{Maz1}, so the Selmer group over the full cyclotomic extension is unable to provide bounds on Selmer coranks at finite layers of the cyclotomic tower. Thus, for supersingular primes, it becomes necessary to consider special restricted plus/minus Selmer groups that turn out to be more amenable to classical arguments. This general program was first proposed and carried out over $\Q$ by Kobayashi in \cite{kobayashi} and has admitted since then many generalizations such as \cite{IP} and \cite{kim09}. It is worth remarking that the $p$-adic analytic counterpart of Kobayashi's theory was provided by Pollack in \cite{pollack-duke}.

In this paper we take up the task of generalizing Bertolini's results in \cite{Ber2} to the case of elliptic curves with good supersingular reduction at $p$. Thus, we consider the complications that arise both from having supersingular reduction and from working over the anticyclotomic $\Z_p$-extension of an imaginary quadratic field satisfying the Heegner hypothesis with respect to the conductor of our elliptic curve.

\subsection{Setup and notation} \label{section:notation}

Throughout this article, $p$ will denote an odd prime number. Let $E/\Q$ be an elliptic curve of conductor $N$ with good supersingular reduction at $p$ and $a_p(E)=0$, where $a_p(E)$ is the $p$-th Fourier coefficient of the weight $2$ newform attached to $E$ by modularity. 

Let $K$ be an imaginary quadratic field such that all the primes dividing $pN$ split in $K$ (in particular, $K$ satisfies the \emph{Heegner hypothesis} relative to $N$), write $\cO_K$ for the ring of integers of $K$ and let 
\[
p\cO_K=\p\p^c
\]
be the factorization of $p$ as a product of (distinct) maximal ideals of $\cO_K$ (here $c$ denotes the nontrivial element of $\Gal(K/\Q)$). Fix algebraic closures $\bar\Q$ and $\bar\Q_p$ of $\Q$ and $\Q_p$, respectively, write $\mathbb C_p$ for the completion of $\bar\Q_p$, choose an embedding $\iota_p:\bar\Q\hookrightarrow\mathbb{C}_p$ and suppose that $\p$ is the prime above $p$ that lands inside the valuation ideal of $\mathbb{C}_p$. Let $T$ denote the $p$-adic Tate module of $E$, set $V:=T\otimes_{\Zp}\Qp$ and $A:=V/T=E[p^\infty]$. For every integer $m\ge 1$, define $A_m:=T/p^m T=E[p^m]$. Moreover, let $G_\Q:=\mathrm{Gal}(\bar{\Q}/\Q)$ be the absolute Galois group of $\Q$.

Let $K_\infty$ be the anticyclotomic $\Z_p$-extension of $K$, whose $n$-th finite layer will be denoted by $K_n$, and let $\Lambda$ be the associated Iwasawa algebra. We assume that the two primes of $K$ above $p$ are totally ramified in $K_\infty$; this is a natural condition to require when working in the supersingular setting (cf., e.g., \cite[Assumptions 1.7, (2)]{DI}, \cite[Hypothesis (S)]{IP}, \cite[Theorem 1.2, (2)]{PolWes}) and holds if $p$ does not divide the class number of $K$. We shall denote the unique primes (of the finite layers) of $K_\infty$ above $\p$ and $\p^c$ by the same symbols. For more details, see Section \ref{sec2}.

Finally, for any compact or discrete $\Lambda$-module $M$ we write $M^\vee:=\Hom_{\Z_p}^\mathrm{cont}(M,\Q_p/\Z_p)$ for its Pontryagin dual, which may be equipped with the compact-open topology (here $\Hom_{\Z_p}^\mathrm{cont}(\bullet,\star)$ denotes continuous homomorphisms of $\Z_p$-modules and $\Q_p/\Z_p$ is equipped with the quotient, i.e., discrete, topology).

 \iffalse
 We write
\[
\TT:=T\otimes \Lambda^\iota,
\]
where $\Lambda^\iota$ denotes  $\Lambda$ equipped with the Galois action given by .....

Note that
\[
H^1(K,\TT_f)\simeq \bigoplus_{v|p}\varprojlim_n H^1(K_{n,v},T_f),
\]
where..... and $H^1(K,\TT_f)$ is free of rank two over $\Lambda$ [REF].
\fi

\subsection{Important hypotheses}\label{sec:hypotheses} In this section, we collect some of the hypotheses that we will impose at various points throughout the paper. We assign labels to these hypotheses for easy recall. Some of the notation appearing in these hypotheses has not been introduced yet; we indicate where this notation is defined, as well as the first time each hypotheses is invoked. 

\begin{itemize}
\item[\textbf{(Heeg)}] Every prime number dividing $Np$ splits in the imaginary quadratic field $K$, where the integer $N$ is understood from context to be the conductor of some elliptic curve over $\Q$ with good supersingular reduction at $p$.

\begin{itemize} \item This setup already appeared in $\S$\ref{section:notation}, and it is assumed throughout the entirety of this paper (for various values of $N$).\end{itemize}

\item[\textbf{(Tam)}] Given an elliptic curve $E/\Q$, the prime $p$ does not divide $ \#\bigl(E / E^0\bigr)$.

\begin{itemize}\item This condition ensures that none of the local Tamagawa factors of $E$ are divisible by $p$. This assumption is necessary in order to obtain an isomorphism in Proposition \ref{prop:control} rather than an injection with bounded cokernel; see, e.g., \cite[Lemma 3.3]{greenberg-cetraro} and the discussion that immediately follows it.
\end{itemize}

\item[\textbf{(Norm)}] For an elliptic curve $E/\Q$ satisfying \textbf{(Tam)}, the  local norm maps $E(K_{n,v})\rightarrow E(K_v)$ are surjective for all $n\ge1$ and all primes of $K$ that do not divide $p$.

\begin{itemize} \item This assumption is introduced in $\S$\ref{sec:admissable} in order to generalize some cohomological lemmas from \cite{BD2}, where the surjectivity of the local norm maps at primes above $p$ is also required (see \cite[Assumption 2.15]{BD2}). However, this additional assumption guarantees that $E$ is $p$-ordinary, which is not within the scope of the present paper.
\end{itemize}

\item[\textbf{(mod $p$)}] For an elliptic curve $E/\Q$ satisfying both $a_p(E)=0$ and $\textbf{(Norm})$, we have $\fE^\pm_n\ne0$ and $\sha^\pm_p(E,K_{n+1}/K_n)=0$  for all $n \in \N$. 
\begin{itemize}\item The notation $\fE^\pm_n$ (which conceals the fact that it depends on $E$) is defined in $\S$\ref{sec:plus-minus-heegner-points}, while $\sha^\pm_p(E,K_{n+1}/K_n)$ is defined in $\S$\ref{sec:computation-of-universal-norms}. This hypothesis first appears in the statement of Theorem \ref{thm:main-no-finite-index}.\end{itemize}

\item[\textbf{(Cong)}] Given two elliptic curves $E_1/\Q$ and $E_2/\Q$, both of which satisfy \textbf{(mod $p$)}, there is an isomorphism $E_1[p] \simeq E_2[p]$ as $G_\Q$-modules. 
\begin{itemize} 
\item For pairs of elliptic curves satisfying this assumption, we will be able to compare the Iwasawa invariants of their plus/minus Selmer groups in $\S$\ref{sec:application-variation}.
\end{itemize}

\end{itemize}

We emphasize that assumption \textbf{(Heeg)} is always in effect and the other assumptions have a cumulative nature: 
\[ 
\textbf{(Cong)} \Rightarrow \textbf{(mod $p$)} \Rightarrow \textbf{(Norm)} \Rightarrow \textbf{(Tam)}.
\]

\subsection{Main results}

For every $n\in\N$ we define plus and minus Mordell--Weil groups $E^\pm(K_n)$, which provide the local conditions in terms of which we introduce our plus and minus Selmer groups $\Sel^\pm_{p^\infty}(E/K_\infty)$ \emph{\`a la} Kobayashi. It is worth remarking that our construction is inspired by Kim (\cite{kim}); in particular, we define plus/minus Selmer groups over the finite layers of $K_\infty/K$ somewhat differently from most of the literature (see, e.g., \cite{ciperiani}, \cite{DI}, \cite{IP}, \cite{LV-BUMI}). This construction has the advantage of yielding a nicer control theorem and also some important duality properties that may be of independent interest.

Our main result, which corresponds to Theorem \ref{thm:main-no-finite-index}, can be stated as follows.

\begin{theorem} \label{thm11}
If \textbf{(Heeg)} and \textbf{(mod $p$)} hold, then the $\Lambda$-module $\Sel_{p^\infty}^\pm(E/K_\infty)$ has no proper $\Lambda$-submodule of finite index.
\end{theorem}

As a sample application of this result, we are able to prove a theorem regarding the variation of Iwasawa invariants among \emph{$p$-congruent} rational elliptic curves, that is, rational elliptic curves with Galois-isomorphic $p$-torsion subgroups. 

\begin{theorem} Assume \textbf{(Heeg)} holds and let $E_1/\Q$ and $E_2/\Q$ be elliptic curves satisfying \textbf{(Cong)}. Then
\[
\mu\bigl(\Sel^\pm_{p^\infty}(E_1/K_{\infty})^\vee\bigr) = 0\;\Longleftrightarrow\; \mu\bigl(\Sel^\pm_{p^\infty}(E_2/K_{\infty})^\vee\bigr)=0.
\]
Furthermore, if these $\mu$-invariants vanish, then the $\lambda$-invariants of the Pontryagin duals of these Selmer groups are related by an explicit formula.
\end{theorem}

This result is in the vein of work by Greenberg--Vatsal (\cite{gv}) and extends to the supersingular setting results of Hatley--Lei in the $p$-ordinary case (\cite{hat-lei2,MRL2}). The reader is referred to the introduction to Section \ref{sec:application-variation} for the definitions of $\mu$- and $\lambda$-invariants, and to Theorem \ref{thm:main2} for the full statement.

\begin{remark} 
Our arguments for obtaining these results have some significant differences from the strategy developed by Greenberg in \cite{greenberg-cetraro}, which has been generalized to many different settings (see, e.g., \cite{hat-lei1}, \cite{hat-lei2}, \cite{kidwell16}, \cite{kim09}, \cite{KO}, \cite{Ponsinet}, \cite{WesIIGD}). In \cite{Greenberg11}, Greenberg showed that his strategy works for very general Selmer groups that can be defined by a surjective global-to-local map in cohomology. In the setting studied in this paper, the plus and minus Selmer groups cannot be defined in this way, since they fail Greenberg's so-called CRK hypothesis. On the other hand, both strategies utilize non-primitive Selmer groups, and the influence of Greenberg's approach is amply evident in our own.
\end{remark}

\subsection{Future directions} 

The original motivation for \cite{Ber2} was to strengthen some of the results from \cite{Ber1}. It would be interesting to study generalizations of results in  \cite{Ber1} in the supersingular setting. For example, on adapting Bertolini's strategy in \cite{Ber2} to our plus/minus case, we have imposed a technical hypothesis on the vanishing of certain relative Shafarevich--Tate groups of $E$ along finite layers of $K_\infty/K$, denoted by $\sha_{p^m}^\pm(E,K_{n+1}/K_n)$ (see Lemma \ref{lem:Un} for more details). It would be worthwhile to investigate this assumption further, e.g., to find sufficient conditions and give explicit examples. 

The results of this paper should also have implications for the annihilators of anticyclotomic plus/minus Selmer groups, in the spirit of \cite{bertolini-lincei}; we plan to tackle this question in a future project.

Finally, the first two named authors recently generalized Bertolini's results to the case of $p$-ordinary modular forms of higher weight \cite{MRL2}; it would be interesting to make a similar study of non-ordinary modular forms.

\subsection*{Acknowledgements} We thank Mirela \c{C}iperiani and Matteo Longo for interesting discussions on some of the topics of this paper. We are also grateful to the anonymous referee for several helpful suggestions which clarified some proofs and otherwise improved the exposition of the paper.

\section{Anticyclotomic Iwasawa algebras} \label{sec2}

%We briefly recall the definition of the anticyclotomic $\Z_p$-extension of $K$ and then introduce the Iwasawa algebra that will be used in the rest of the paper.

\subsection{The anticyclotomic $\Z_p$-extension of $K$} \label{anticyclotomic-subsec}

For every $m\in\N$ let $H_{p^m}$ denote the ring class field of $K$ of conductor $p^m$, then set $H_{p^\infty}:=\cup_{m\in\N}H_{p^m}$. There is an isomorphism
\[ \Gal(H_{p^\infty}/K)\simeq\Z_p\times\Delta, \]
where $\Delta$ is a finite group. The \emph{anticyclotomic $\Z_p$-extension} $K_\infty/K$ is the unique $\Z_p$-extension of $K$ contained in $H_{p^\infty}$. It can be characterized as the unique $\Z_p$-extension of $K$ that is Galois and non-abelian (in fact, generalized dihedral) over $\Q$. We can write $K_\infty:=\cup_{m\geq0}K_n$, where $K_n$ is the unique subfield of $K_\infty$ such that
\[ G_n:=\Gal(K_n/K)\simeq\Z/p^n\Z. \]
In particular, $K_0=K$. We also define
\[ G_\infty:=\varprojlim_m G_m=\Gal(K_\infty/K)\simeq\Z_p, \]
where the inverse limit is taken with respect to the natural restriction maps. Finally, for all $n,n'\in\N\cup\{\infty\}$ with $n\leq n'$ we set
\[
\cG_{n'/n}:=\Gal(K_{n'}/K_n).
\]
In particular, $\cG_{n/0}=G_n$ for all $n\in\N$.

\subsection{The Iwasawa algebra $\Lambda$} \label{iwasawa-algebra-subsec}

Throughout our article, we fix a topological generator $\gamma_\infty$ of $G_\infty$. Furthermore, we consider the Iwasawa algebra
\[ \Lambda:=\varprojlim_n\Zp[G_n]=\Z_p[\![G_\infty]\!] \]
attached to $K_\infty/K$, the inverse limit being taken with respect to the maps that are induced by restriction. As is well known (see, e.g., \cite[Proposition 5.3.5]{NSW}), $\gamma_\infty$ determines an isomorphism of topological $\Z_p$-algebras $\Lambda \overset\simeq\longrightarrow\Z_p[\![X]\!]$ such that $\gamma_\infty\mapsto 1+X$. In the rest of the paper, we shall tacitly identify $\Lambda$ with $\Z_p[\![X]\!]$ in this way.

\section{Selmer groups and control theorems}

\subsection{Plus and minus Selmer groups over $K_\infty$} \label{plus/minus-subsec}

We introduce plus and minus norm groups \emph{\`a la} Kobayashi (see \cite{kobayashi}). First of all, define the two sets of indices
\begin{align*}
    S_n^+&:=\left\{0,1,\ldots, n\right\}\cap 2\ZZ;\\
    S_n^-&:=\left\{0,1,\ldots, n\right\}\cap (2\ZZ+1).
\end{align*}
Let $n\geq1$ be an integer. For $v\in\{\p,\p^c\}$, set
\[
\hat E^\pm(K_{n,v}):=\left\{P\in \hat E(\fM_{K_{n,v}})\;\Big|\;\Tr_{n/m+1}(P)\in \hat E(\fM_{K_{m,v}})\text{ for all }m\in S_n^\pm \right\},
\]
where we write $\hat E(\fM_{K_{n,v}})$ for the formal group of $E$ whose points are defined over the maximal ideal $\fM_{K_{n,v}}$ of the valuation ring of $K_{n,v}$ and 
\[ \Tr_{n/m+1}:\hat E(\fM_{K_{n,v}})\longrightarrow\hat E(\fM_{K_{m+1,v}}) \]
is the trace map on formal groups.

Let $v\in\{\p,\p^c\}$. Following \cite[\S3.3]{kim}, we define 
\begin{equation} \label{defn:kimpm}
\HH^\pm_v:=\bigcup_{n\ge 0} \hat E^\pm(K_{n,v})\otimes\Qp/\Zp.
\end{equation}
For $n\in\N$, we also set
\[
\HH^\pm_{n,v}:=\left(\HH^\pm_v\right)^{\Gal(K_{\infty,v}/K_{n,v})}.
\]

\begin{remark} \label{rmk:kummer-image}
Using the Kummer map and the fact (shown in the proof of Lemma \ref{lem:inj-p} below) that the natural maps
$$H^1(K_n , E[p^\infty ]) \rightarrow H^1(K_{\infty,v} , E[p^\infty ])^{\cG_{\infty/n}}$$
and
$$H^1 (K_{n,v} , E[p^m ]) \rightarrow H^1 (K_{n,v} , E[p^\infty])[p^m ]$$
are isomorphisms, we may identify $\HH^\pm_v$ as a $\Lambda$-submodule of $H^1(K_{\infty,v},A):=\varinjlim_n H^1(K_{n,v},A)$. In turn, we may identify $\HH^\pm_{n,v}$ and $\HH^\pm_{n,v}[p^m]$ as submodules of $H^1(K_{n,v},A)$ and $H^1(K_{n,v},A_m)$, respectively.
\end{remark}

In line with our general notational conventions, let $(\HH_v^\pm)^\vee$ be the Pontryagin dual of $\HH_v^\pm$.

\begin{proposition}\label{prop:kim}
\begin{enumerate}[label=(\alph*)]
    \item The $\Lambda$-module $(\HH_v^\pm)^\vee$ is free of rank one.
    \item Let $m,n\in\N$. Under the local Tate pairing 
    \[ H^1(K_{n,v},A_m)\times H^1(K_{n,v},A_m)\longrightarrow \ZZ/p^m\ZZ, \]
    the exact annihilator of $\HH_{n,v}^\pm[p^m]$ is $\HH_{n,v}^\pm[p^m]$.
\end{enumerate}
\end{proposition}

\begin{proof}
Part (a) is \cite[Proposition~3.13]{kim} when the sign is $-$. This has been subsequently generalized to the $+$ case in \cite[Proposition~2.11]{kimZp2}, as our ground field is $\Qp$.  Part (b) then follows from the proof of \cite[Proposition~3.15]{kim}.
\end{proof}

Now we define Selmer groups for $E$ over $K_\infty$ and over the finite layers of $K_\infty/K$. First of all, for every $n\in\N$ let $S_n$ be the set of places of $K_n$ dividing $Np\infty$, let $K_{n,S_n}$ be the maximal extension of $K_n$ unramified outside $S_n$ and write $H^1_{S_n}(K_n,\star)$ as a shorthand for $H^1(K_{n,S_n}/K_n,\star)$. In the following, it is convenient to set also $A_\infty:=A=E[p^\infty]$.
 
\begin{definition}\label{def:selmer-basic}
Let $m\in\N\cup\{\infty\}$ and $n\in\N$. The \emph{$p^m$-Selmer group of $E$ over $K_n$} is 
\[
\Sel_{p^m}(E/K_n):=\ker\Bigg(H_{S_n}^1(K_n,A_m)\longrightarrow \prod_{v \in S_n}\frac{H^1(K_{n,v},A_m)}{E(K_{n,v})/p^m}\Bigg).
\]
The \emph{$p^m$-Selmer group of $E$ over $K_\infty$} is 
\[
\Sel_{p^m}(E/K_\infty):=\varinjlim_n\Sel_{p^m}(E/K_n),
\]
the direct limit being taken with respect to the restriction maps.
\end{definition}

We also introduce plus/minus Selmer groups \emph{\`a la} Kobayashi, along the lines of \cite{kim}.

\begin{definition}
Let $m\in\N\cup\{\infty\}$ and $n\in\N$. The \emph{plus/minus Selmer groups of $E$ over $K_n$} are
\[
\Sel_{p^m}^\pm(E/K_n):=\ker\Bigg(H_{S_n}^1(K_n,A_m)\longrightarrow \prod_{v \in S_n, v\nmid p}\frac{H^1(K_{n,v},A_m)}{E(K_{n,v})/p^m}\times \prod_{v \mid p}\frac{H^1(K_{n,v},A_m)}{\HH^\pm_{n,v}[p^m]}\Bigg).
\]
The \emph{plus/minus Selmer groups of $E$ over $K_\infty$} are
\[
\Sel_{p^m}^\pm(E/K_\infty):=\varinjlim_n\Sel_{p^m}^\pm(E/K_n),
\]
the direct limit being taken with respect to the restriction maps.
\end{definition}

\begin{remark}
Our definitions are different from those in \cite{IP} and \cite{kobayashi}, unless the base field is $K_0=K$ or $K_\infty$. We have followed \cite[\S4.4]{kim} because we would like our Selmer conditions at $p$ to satisfy part (b) of Proposition~\ref{prop:kim} for our applications later. 
\end{remark}

\begin{remark}
If $m=1$, then we omit the superscript and simply write, e.g., $\Sel_p^\pm(E/K_n)$.
\end{remark}

The next result deals with the $n=0$ case.

\begin{lemma}\label{lem:same}
For all $m\in\N\cup\{\infty\}$, there is an equality
\[
\Sel_{p^m}^\pm(E/K)=\Sel_{p^m}(E/K).
\]
\end{lemma}

\begin{proof}
In light of Remark \ref{rmk:kummer-image}, it suffices to show that $\HH^\pm_{0,v}[p^m]$ coincides with the image of $E(K_v)/p^mE(K_v)$ in $H^1(K_v,A_m)$ under the Kummer map. Indeed, $\HH^\pm_{0,v}[p^m]$, $E(K_v)/p^mE(K_v)$ and $H^1(K_v,A_m)\big/\bigl(E(K_v)/p^mE(K_v)\bigr)$ are all free of rank one over $\Z/p^m\Z$. On the other hand, $E(K_v)/p^mE(K_v)$ is contained in $\HH^\pm_{0,v}[p^m]$, as explained in the proof of \cite[Proposition~3.15]{kim}, and the result follows.
\end{proof}

\subsection{Control theorems}

Our goal is to prove a control theorem for the plus and minus Selmer groups in our context (Proposition \ref{prop:control}) generalizing the corresponding result in the ordinary case (see, e.g., \cite[\S 2.3]{Ber1}).

\begin{lemma}\label{lem:inj-p}
Let $v\in\{\p,\p^c\}$ and $m,n,n'\in\N\cup\{\infty\}$ with $n\leq n'$. The natural maps
\begin{equation}\label{eq:inj}
\frac{H^1(K_{n,v},A_m)}{\HH^\pm_{n,v}[p^m]}\longrightarrow \frac{H^1(K_{n',v},A_m)}{\HH^\pm_{n',v}[p^m]},\qquad \frac{H^1(K_{n,v},A_m)}{\HH^\pm_{n,v}[p^m]}\longrightarrow\frac{H^1(K_{n,v},A)}{\HH^\pm_{n,v}}
\end{equation}
are injective.
\end{lemma}

\begin{proof} Since $E$ has supersingular reduction at $p$, $H^0(K_{n',v},A_m)=0$ by the proof of \cite[Proposition~7]{kobayashi} (see also \cite[Lemma~4.6]{IP}). The inflation-restriction exact sequence gives an isomorphism
\begin{equation}\label{eq:res-isom} \res:H^1(K_{n,v},A_m)\overset\simeq\longrightarrow H^1(K_{n',v},A_m)^{\cG_{n'/n}}. \end{equation}
In particular, it gives an injection $H^1(K_{n,v},A_m)\hookrightarrow H^1(K_{n',v},A_m).$ To show that the first map in \eqref{eq:inj} is injective, it is enough to show that
\begin{equation}\label{eq:invariant}
\HH^\pm_{n,v}[p^m]={\HH^\pm_{n',v}[p^m]}^{\cG_{n'/n}}.
\end{equation}
But Proposition \ref{prop:kim}(a) tells us that there is an isomorphism of $\Lambda$-modules
\[ \HH^\pm_{n,v}[p^m]\simeq{(\Z/p^m\Z)[G_n]}^\vee, \]
and similarly for $n'$. Thus \eqref{eq:invariant} holds.

Now we study the second map in \eqref{eq:inj}. Consider the short exact sequence
\[
0\longrightarrow A_m\longrightarrow A\xrightarrow{p^m\cdot}A\longrightarrow 0.
\]
The fact that  $H^0(K_{n,v},A)=0$ implies that
\begin{equation} \label{eq:AMpm}
H^1(K_{n,v},A_m)=H^1(K_{n,v},A)[p^m],
\end{equation}
which gives the second injection.
\end{proof}

We are now in a position to prove a strong control theorem for our plus/minus Selmer groups. In order to obtain an isomorphism (rather than an injection with bounded cokernel) we need the following hypothesis.
\begin{itemize}
\item[\textbf{(Tam)}] Given an elliptic curve $E/\Q$, the prime $p$ does not divide $\#\bigl(E/ E^0\bigr)$.
\end{itemize}
\begin{proposition}\label{prop:control} Let $E/\Q$ be an elliptic curve satisfying \textbf{(Tam)}. Let $m,m',n,n'\in\N\cup\{\infty\}$ with $m\leq m'$ and $n\leq n'$. The restriction map induces an isomorphism of $\Lambda$-modules
\[
\Sel_{p^m}^\pm (E/K_n)\simeq {\Sel_{p^{m'}}^\pm(E/K_{n'})[p^m]}^{\cG_{n'/n}}.
\]
\end{proposition}

\begin{proof} 
Consider the commutative diagram
\[
\xymatrix{
0\ar[r]&\Sel_{p^m}^\pm (E/K_{n'})\ar[r]\ar[d]& H_{S_{n'}}^1(K_{n'},A_m)\ar[r]\ar[d]& \prod_{v \in S_{n'}}\frac{H^1(K_{n',v},A_m)}{\HH^\pm_{n',v}[p^m]}\ar[d]^{\delta_v^\pm}\\
0\ar[r]&\Sel_{p^{m'}}^\pm (E/K_{n'})[p^m]\ar[r]& H_{S_{n'}}^1(K_{n'},A_{m'})[p^m]\ar[r]& \left(\prod_{v \in S_{n'}}\frac{H^1(K_{n',v},A_{m'})}{\HH^\pm_{n',v}[p^{m'}]}\right)[p^m],}
\]
where we have written $\HH^\pm_{n',v}[p^m]$ and $\HH^\pm_{n',v}[p^{m'}]$ for $E(K_{n',v})/p^{m}$ and $E(K_{n',v})/p^{m'}$ respectively when $v\nmid p$. The middle map is induced by taking cohomology of the short exact sequence
\[
0\longrightarrow A_m\longrightarrow A_{m'}\stackrel{p^m\cdot }\longrightarrow A_{m'-m}\longrightarrow 0.
\] As in the proof of Lemma~\ref{lem:inj-p}, we have $H^0(K_{n'},A_{m'-m})=0$. In particular, the middle vertical map is an isomorphism. When $v|p$, Lemma~\ref{lem:inj-p} says that $\delta_v^\pm$ is injective. When $v\nmid p$, the local condition for the plus/minus Selmer group is the same as that for the ``full'' Selmer group, so the injectivity follows from \cite[\S 2.3, Lemma 1, (1)]{Ber1}. The snake lemma then shows that the first vertical map is an isomorphism as claimed.

Now, consider the commutative diagram
\[
\xymatrix{
0\ar[r]&\Sel_{p^m}^\pm (E/K_n)\ar[r]\ar[d]& H_{S_n}^1(K_n,A_m)\ar[r]\ar[d]& \prod_{v \in S_n}\frac{H^1(K_{n,v},A_m)}{\HH^\pm_{n,v}[p^m]}\ar[d]^{\gamma_v^\pm}\\
0\ar[r]&\Sel_{p^m}^\pm (E/K_{n'})^{\cG_{n'/n}}\ar[r]& H_{S_{n'}}^1(K_{n'},A_m)^{\cG_{n'/n}}\ar[r]& \left(\prod_{v \in S_{n'}}\frac{H^1(K_{n',v},A_m)}{\HH^\pm_{n',v}[p^m]}\right)^{\cG_{n'/n}},}
\]
where we have written $\HH^\pm_{n,v}[p^{m}]$ and $\HH^\pm_{n',v}[p^{m}]$ for $E(K_{n,v})/p^m$ and $E(K_{n',v})/p^{m}$ respectively when $v\nmid p$. As before, $H^0(K_{n'},A_m)=0$. Thus, the inflation-restriction exact sequence says that the middle map is an isomorphism. When $v|p$, Lemma~\ref{lem:inj-p} says that $\gamma_v^\pm$ is injective. For $v \nmid p$, it is well-known (see e.g. \cite[Lemma 3.3]{greenberg-cetraro} and the discussion immediately following it) that the kernel of $\gamma_v^\pm$ is trivial under hypothesis \textbf{(Tam)}. We conclude once again by the snake lemma.
\end{proof}

\section{Plus and minus Heegner points} \label{plus/minus-sec}

\subsection{Plus/minus Mordell--Weil groups}

Let $n\geq1$ be an integer. We define plus and minus norm subgroups of $E(K_n)$ that are global counterparts of the plus/minus norm groups introduced in \S \ref{plus/minus-subsec}. For all $m \in S_n^\pm$ with $m<n$, let 
\[ \Tr_{n/m+1}:E(K_n)\longrightarrow E(K_{m+1}) \]
be the Galois trace map with respect to the group law on $E$.

\begin{definition} \label{plus-minus-def}
The \emph{plus/minus Mordell--Weil groups} of $E$ over $K_n$ are
\[ E^\pm(K_n):=\bigl\{P\in E(K_n)\mid\text{$\Tr_{n/m+1}(P)\in E(K_{m})$ for all $m\in S_n^\pm$ with $m<n$}\bigr\}. \]
\end{definition}

Similar to \eqref{defn:kimpm}, we also define 
\[ \HH^\pm_\infty:=\bigcup_{n\ge 0}E^\pm(K_n)\otimes\Qp/\Zp,\qquad \HH^\pm_n:=\bigl(\HH^\pm_\infty\bigr)^{\Gal(K_\infty/K_n)}. \]

\begin{remark}
The $\Lambda$-module $\HH^\pm_\infty$ (respectively, $\HH^\pm_n$) may be identified with a submodule of $\Sel_{p^\infty}^\pm(E/K_\infty)$ (respectively, $\Sel_{p^\infty}^\pm(E/K_n)$) via the usual Kummer map in Galois cohomology. Analogously, given an integer $m\ge1$, we may view $\HH^\pm_n[p^m]$ as a $\Lambda$-submodule of $\Sel_{p^m}^\pm(E/K_n)$.
\end{remark}

\subsection{Plus/minus Heegner points}\label{sec:plus-minus-heegner-points}

Let $\bigl\{z_n\in E(K_n)\bigr\}_{n\geq1}$ be a compatible family of Heegner points as in \cite[\S 4.2]{LV-BUMI}. Following \cite[\S 4.3]{LV-BUMI}, we give

\begin{definition}
The \emph{plus/minus Heegner points} are
\[ z^+_n:=\begin{cases}z_n&\text{if $n$ is even},\\[3mm]z_{n-1}&\text{if $n$ is odd},\end{cases}\qquad z^-_n:=\begin{cases}z_{n-1}&\text{if $n$ is even},\\[3mm]z_n&\text{if $n$ is odd}.\end{cases}
\]
\end{definition}

Since we have assumed that $a_p(E)=0$, it is a consequence of the formulas in \cite[\S 3.1, Proposition 1]{Perrin-Riou} that the points $z_m^\pm$ satisfy the following relations:
\begin{enumerate}[label=(\alph*)]
   \item$\tr_{K_m/K_{m-1}}(z_m^+)=-z_{m-1}^+$\quad for every even $m\geq2$;
\item $\tr_{K_m/K_{m-1}}(z_m^+)=pz_{m-1}^+$\quad for every odd $m\geq1$;
\item $\tr_{K_m/K_{m-1}}(z_m^-)=pz_{m-1}^-$\quad for every even $m\geq2$;
\item $\tr_{K_m/K_{m-1}}(z_m^-)=-z_{m-1}^-$\quad for every odd $m\geq3$;
\item $\tr_{K_1/K_0}(z_1^-)=\frac{p-1}{2}z_0^-=\frac{p-1}{2}z_0$.
\end{enumerate}
In particular, we see that $z_m^\pm \in E^\pm(K_m)$. 

For all $m,n\in\N$, set $R_{m,n}:=(\ZZ/p^m\ZZ)[G_n]$. As in \cite[\S4.4]{LV-BUMI}, we define $\cE_{m,n}^\pm$ to be the $R_{m,n}$-submodule of $\Sel_{p^m}^\pm(E/K_n)$ generated by $z^\pm_n$. This in turn defines  $\Lambda$-submodules 
\[
\cE^\pm_\infty:=\varinjlim_m\cE_{m,m}^\pm\subset\Sel_{p^\infty}^\pm(E/K_\infty)
\]
as well as the Pontryagin duals
\[
\cH_\infty^\pm:=(\cE_\infty^\pm)^\vee=\varprojlim_m(\cE_{m,m}^\pm)^\vee.
\]
Finally, we introduce the $R_{m,n}$-module
\[
\fE_{m,n}^\pm:=(\cE_{\infty}^\pm)^{\Gal(K_{\infty}/K_n)}[p^m].
\]
When $m=1$, we shall omit the index $m$ from the notation and simply write $R_n$, $\cE_n^\pm$, $\fE_n^\pm$.

We recall the following result on $\cH_\infty^\pm$.

\begin{proposition}
The $\Lambda$-module $\cH_\infty^\pm$ is finitely generated, torsion-free and of rank 1.
\end{proposition}

\begin{proof} This is \cite[Proposition~4.7]{LV-BUMI}. \end{proof}

We strengthen this slightly in the following proposition.

\begin{proposition} \label{free-prop}
The $\Lambda$-module $\cH_\infty^\pm$ is free of rank one.
\end{proposition}

\begin{proof}
By definition, $\cE_{m,m}^\pm$ is a cyclic $R_{m,m}$-module. Thus, $\varprojlim (\cE_{m,m}^\pm)^\vee$ is a cyclic $\varprojlim R_{m,m}$-module. Thus, $\cH^\pm_\infty$ is cyclic over $\Lambda$. Since it is also torsion-free, and since $\Lambda$ is an integral domain, this implies $\cH^\pm_\infty$ is free of rank $1$.
\end{proof}

We deduce 

\begin{corollary}\label{cor:cyclic}
The $R_n$-module $\fE_n^\pm$ is cyclic.
\end{corollary}

\begin{proof} Immediate from Proposition \ref{free-prop}. \end{proof}

\begin{lemma}\label{lem:Hembeds}
There is an injection of $R_n$-modules
\[
\fE_{m,n}^\pm\longmono\Sel_{p^m}^\pm(E/K_n).
\]
\end{lemma}
\begin{proof}
We may identify $\fE^\pm_{m,n}$ with a submodule of $H^1(K_n,A_m)$ via the Kummer map. On the other hand, the inclusion $\cE_\infty^\pm\subset \HH_\infty^\pm$ induces an inclusion $\fE^\pm_{m,n}\subset\HH^\pm_n[p^m]$, and the lemma is proved. \end{proof}

From now on we shall view $\fE_{m,n}^\pm$ as a submodule of $\Sel_{p^m}^\pm(E/K_n)$ via Lemma~\ref{lem:Hembeds} without further notice.

\begin{lemma}\label{lem:embeds}
Let $n\ge1$ be an integer. There is a natural injection $\fE^\pm_{m,n}\hookrightarrow\fE_{m,n+1}^\pm$ whose image is $\bigl(\fE_{m,n+1}^\pm\bigr)^{\cG_{n+1/n}}$.
\end{lemma}

\begin{proof}
As in the proof of Lemma~\ref{lem:Hembeds}, we may identify $\fE^\pm_{m,n}$ and $\fE^\pm_{m,n+1}$ as submodules of $H^1(K_n,A_m)$ and $H^1(K_{n+1},A_m)$, respectively. Furthermore, as in the proof of Lemma~\ref{lem:inj-p}, the inflation-restriction exact sequence gives an isomorphism
\[
H^1(K_n,A_m)\simeq H^1(K_{n+1},A_m)^{\cG_{n+1/n}},
\]
and the result follows.
\end{proof}

\section{On proper $\Lambda$-submodules of finite index} 
\subsection{Admissible classes}\label{sec:admissable}
We prove the supersingular analogues of \cite[Theorem~3.2]{BD2}, which will serve as one of the key ingredients in the proof of Proposition~\ref{prop:tatepairing} below. Throughout, we fix integers $m,n\in\N$. We also assume that the hypothesis \textbf{(Tam)} holds. 

\begin{definition}
 A prime $v$ of $K$ is \emph{admissible} for $(E,K_n,p^m)$ if
  \begin{enumerate}
      \item $E$ has good reduction at $v$;
      \item $v$ does not divide $p$;
      \item $v$ splits completely in $K_n/K$;
      \item the group $E(K_v)/p^m$ is isomorphic to $(\ZZ/p^m)^2$.
  \end{enumerate}

 A set $\Sigma$ of primes of $K$ is \emph{admissible} for $(E,K_n,p^m )$ if 
  \begin{enumerate}
      \item all $v\in \Sigma$ are admissible for $(E,K_n,p^m)$;
      \item the map $\Sel_{p^m}(E/K)\rightarrow \prod_{v\in \Sigma}E(K_v)/p^m$ is injective.
  \end{enumerate}
\end{definition}
We recall from \cite[Lemma~2.23]{BD2} that admissible sets exist. In the rest of this section, we fix an admissible set $\Sigma$. We give the following supersingular analogues of admissible classes defined in \cite[Definition~2.24]{BD2}.

\begin{definition}
  The group of plus/minus admissible classes with respect to $\Sigma$ in $H^1(K_n,A_m)$ is
  \[
  H^1_{\Sigma,\pm}(K_n,A_m)=\ker\Biggl(H^1(K_n,A_m)\longrightarrow \prod_{v\notin \Sigma,v\nmid p}\frac{H^1(K_{n,v},A_m)}{E(K_{n,v})/p^m}\times \prod_{v|p}\frac{H^1(K_{n,v},A_m)}{\HH^\pm_{n,v}[p^m]}\Biggr).
  \]
\end{definition}

\begin{remark}\label{rk:admissible}
Note that $H^1_{\Sigma,\pm}(K_n,A_m)\subset H^1_{S_n}(K_n,A_m)$, as $\Sigma$ contains no bad primes.
\end{remark}

We can  generalize Lemmas~2.25, 2.26, 2.27 and 3.1 of \cite{BD2} to the supersingular setting.

\begin{lemma}\label{lem:BD}
\begin{itemize}
  \item[(a)]The map $\Sel_{p^m}^\pm(E/K_n)\rightarrow \bigoplus_{v\in\Sigma} E(K_{n,v})/p^m$ is injective.
\item[(b)]The map $\prod_{v\in \Sigma}\frac{H^1(K_{n,v},A_m)}{E(K_{n,v})/p^m}\rightarrow \Hom(\Sel_{p^m}^\pm(E/K_n),\ZZ/p^m)$ is surjective.
\item[(c)] The restriction map $H^1_{\Sigma,\pm}(K,A_m)\rightarrow H^1_{\Sigma,\pm}(K_n,A_m)^{\cG_{n/0}}$ is an isomorphism.
\item[(d)] There is an exact sequence
\[
0\longrightarrow \Sel_{p^m}^\pm(E/K_n)\longrightarrow H^1_{\Sigma,\pm}(K_n,A_m)\longrightarrow\prod_{v\in \Sigma}\frac{H^1(K_{n,v},A_m)}{E(K_{n,v})/p^m}\longrightarrow \Sel_{p^m}^\pm(E/K_n)^\vee\longrightarrow 0.
\]
\end{itemize}

\end{lemma}
\begin{proof}
Part (a) follows from the same proof as  \cite[Lemma~2.25]{BD2} once we replace the control theorem in the ordinary case (\cite[Lemma~2.19]{BD2}) by Proposition~\ref{prop:control}.

Part (b) follows from part (a), as in the proof of \cite[Lemma~2.26]{BD2}; namely, this is dual to property (2) in the definition of an admissible set of primes.

Part (c) can be proved in the same way as Proposition~\ref{prop:control}. 

Part (d) now follows from part (b) as in the proof of \cite[Lemma~3.1]{BD2}.
\end{proof}

This allows us to prove the following generalization of \cite[Theorem~3.2]{BD2}.
\begin{theorem}\label{thm:BD}
The $R_{m,n}$-module $H^1_{\Sigma,\pm}(K_n,A_m)$ is isomorphic to $\left(R_{m,n}\right)^{2|\Sigma|}$.
\end{theorem}
\begin{proof}
This follows from the same proof as \cite[Theorem~3.2]{BD2}, where we replace the appropriate lemmas in \cite{BD2} with their counterparts given in Lemma~\ref{lem:BD}.
\end{proof}

\subsection{Cohomology, universal norms and perfect pairings}

For all $m,n\in\N$, write $I_{m,n}$ for the augmentation ideal of $R_{m,n}$. The following result generalizes \cite[Proposition~6.3]{Ber2} to our plus/minus setting under the following hypothesis.
\begin{itemize}
\item[\textbf{(Norm)}] The local norm maps
    \[ E(K_{n,v})\longrightarrow E(K_v) \]
    are surjective for all $n\ge1$ and all primes of $K$ that are coprime to $p$.
\end{itemize}
Note that \textbf{(Norm)} is slightly weaker than \cite[Assumption 4]{Ber2} and \cite[Assumption 2.15]{BD2}.

\begin{proposition} \label{prop:tatepairing}
Let $m,n\in\N$. Suppose that \textbf{(Norm)} holds. There is a perfect pairing of Tate cohomology groups
\[
{\langle\cdot,\cdot\rangle}_{m,n} : \hat{H}^0\bigl(G_n,\Sel_{p^m}^\pm(E/K_n)\bigr)\times \hat{H}^{-1}\bigl(G_n,\Sel_{p^m}^\pm(E/K_n)\bigr)\longrightarrow I_{m,n}/I_{m,n}^2.
\]
\end{proposition}

\begin{proof}
We fix an admissible set of primes for $(E,K_n,p^m)$, which we denote by $\Sigma$. We write $\Sel_{p^m}^\pm(E/K_n)^0$ for the kernel of the corestriction map $\Sel_{p^m}^\pm(E/K_n)\rightarrow\Sel_{p^m}^\pm(E/K)$. Fix a generator $\gamma_n$ of $G_n$. Recall from Lemma~\ref{lem:BD}(d) that $\Sel_{p^m}^\pm(E/K_n)^0$  is contained in $H^1_{\Sigma,\pm}(K_n,A_m)$, which is a free $R_{m,n}$-module by Theorem~\ref{thm:BD}. In particular, the kernel of the corestriction  $H^1_{\Sigma,\pm}(K_n,A_m)\rightarrow H^1_{\Sigma,\pm}(K,A_m)$ is $(\gamma_n-1)H^1_{\Sigma,\pm}(K_n,A_m)$. Therefore if $y\in \Sel_{p^m}^\pm(E/K_n)^0$, there exists 
$z\in H^1_{\Sigma,\pm}(K_n,A_m)$ satisfying  
\begin{equation}
    y=(\gamma_n-1)z.\label{eq:augmentation}
\end{equation}
Let $z_v$ denote the image of $z$ in $H^1(K_{n,v},A_m)/\HH^\pm_{n,v}[p^m]$, where $v$ is a place of $K_n$ and we denote $E(K_{n,v})/p^m$ by $\HH^\pm_{n,v}[p^m]$ when $v\nmid p$. Let us denote the pairing induced by  local Tate duality by
\[
[\cdot,\cdot]_v:\HH^1_{n,v}[p^m]\times H^1(K_v,A_m)/\HH^1_{n,v}[p^m]\longrightarrow \ZZ/p^m\ZZ.
\]
The definition of $y$ and \eqref{eq:augmentation} together imply that 
\[
z_v\in\left( H^1(K_{n,v},A_m)/\HH^\pm_{n,v}[p^m]\right)^{G_n}.
\]
When $v\nmid p$, the surjectivity of the norm map given by \textbf{(Norm)} implies that the restriction map induces an isomorphism
\begin{equation}
    \left( H^1(K_{n,v},A_m)/\HH^\pm_{n,v}[p^m]\right)^{G_n}\cong H^1(K_{0,v},A_m)/\HH^\pm_{0,v}[p^m]
\label{eq:res-iso}
\end{equation}
since the restriction map on these quotients is dual to the norm map $$E(K_{n,v})/p^m\rightarrow E(K_v)/p^m$$
as given by $[\cdot,\cdot]_v$. The same is true when $v|p$ since $(\HH_v^\pm)^\vee$ is free of rank one over $\Lambda$ by Proposition~\ref{prop:kim}(a).

We define a pairing
\[
{[\cdot,\cdot]}_{m,n}:\Sel_{p^m}^\pm(E/K)\times \Sel_{p^m}^\pm(E/K_n)^0\longrightarrow I_{m,n}/I_{m,n}^2
\]
by sending $(x,y)$ to
\[
\sum_{v}{[x_v,z_v]}_v(\gamma_n-1)\mod I_{m,n}^2,
\]
where the sum runs over all primes of $K$, $x_v$ is the natural image of $x$ in $\HH^\pm_{n,v}[p^m]$ and $z_v$ is identified with its image in $H^1(K_v,A_m)/\HH^\pm_{0,v}[p^m]$ as given by the isomorphism \eqref{eq:res-iso}. 
The argument in \cite[Proposition~6.3]{Ber2} to show that this pairing is independent of the choice of $z$ or $\gamma_n$ carries over to our setting, as it relies on algebraic properties of Galois cohomology only. 

To check that ${[\cdot,\cdot]}_{m,n}$ induces a perfect pairing ${\langle\cdot,\cdot\rangle}_{m,n}$ on the Tate cohomology groups it is enough to show that we have a right non-degenerate pairing, as the two groups have the same order. We extend the proof of \cite[Lemma~6.15]{Milne} to our setting. Consider the commutative diagram
\[
\xymatrix{
\bigoplus_{v\in S_n}H^1(K_{n,v},A_m)\ar[d]\\
\bigoplus_{v\in S_n}\frac{H^1(K_{n,v},A_m)}{\HH^\pm_{n,v}[p^m]}&H^1_{S_n}(K_n,A_m)\ar[l]\ar[ul]&\Sel^\pm_{p^m}(E/K_n)\ar[l]&\ar[l]0.}
\]
 If we take $\Z/p^m\Z$-linear duals, then Proposition~\ref{prop:kim} gives us a commutative diagram
\[
\xymatrix{
\bigoplus_{v\in S_n}H^1(K_{n,v},A_m)\ar[dr]\\
\bigoplus_{v\in S_n}\HH^\pm_{n,v}[p^m]\ar[u]\ar[r]&{H^1_{S_n}(K_n,A_m)}^*\ar[r]&{\Sel^\pm_{p^m}(E/K_n)}^*\ar[r]&0.}
\]
Let $y$ be an element of the right kernel of ${[\cdot,\cdot]}_{m,n}$ and let $z\in H^1_{\Sigma,\pm}(K_n,A_m)$ be as in \eqref{eq:augmentation}. By Remark~\ref{rk:admissible}, we may consider it as an element of $H^1_{S_n}(K_n,A_m)$. As can be checked by a diagram chase, the image of $z$ in $\bigoplus_{v\in S_n} H^1(K_{n,v},A_m)$ splits as $a_1+a_2$, where $a_1$ is in the image of $\bigoplus_{v\in S_n} \HH^\pm_{n,v}[p^m]$ and $a_2$ is in the image of $H^1_{S}(K,A_m)$ under the restriction map composed with the localization maps. Therefore $z$ itself decomposes as $z_1+z_2$, where $z_1\in \Sel_{p^m}^\pm(E/K_n)$ and $z_2\in \image\bigl(H^1_{S}(K,A_m)\hookrightarrow H^1_{S_n}(K_n,A_m)\bigr)$. This shows that $y\in (\gamma_n-1)\Sel_{p^m}^\pm(E/K_n)$. In particular, the image of $y$ in $\hat H^{-1}\bigl(G_n,\Sel_{p^m}^\pm(E/K_n)\bigr)$ is zero. This shows that the pairing ${\langle\cdot,\cdot\rangle}_{m,n}$ is right non-degenerate, as required.
\end{proof}

For every $n\in\N$, define
\[
S_p^\pm(E/K_n):=\varprojlim_m \Sel_{p^m}^\pm(E/K_n),
\]
the inverse limit being taken with respect to the multiplication-by-$p$ maps $E[p^{m+1}] \overset{p\cdot}\longrightarrow E[p^m]$. For all $n,n'\in\N$ with $n\leq n'$ there is a natural corestriction map 
\begin{equation} \label{cores-eq}
\cor_{K_{n'}/K_n}: S_p^\pm(E/K_{n'})\longrightarrow S_p^\pm(E/K_n).
\end{equation}
We can then define the $\Lambda$-module
\[ \hat S^\pm_p(E/K_\infty):=\varprojlim_n S_p^\pm(E/K_n), \]
the inverse limit being taken with respect to the corestriction maps in \eqref{cores-eq}.

The next proposition will be used in the proof of our main result (Theorem \ref{thm:main-no-finite-index}).

\begin{proposition} \label{prop:PR}
There is a canonical isomorphism of $\Lambda$-modules
\[
\hat S^\pm_p(E/K_\infty)\simeq\Hom_\Lambda\bigl(\Sel_{p^\infty}^\pm(E/K_\infty)^\vee,\Lambda\bigr).
\]
\end{proposition}

\begin{proof}
One can proceed as in the proof of \cite[Lemme~5]{Perrin-Riou}, replacing the control theorem used in \cite{Perrin-Riou} with Proposition~\ref{prop:control}.
\end{proof}

The \emph{universal norm submodule} of $S^\pm_p(E/K)$ is
\[
US_p^\pm(E/K):=\bigcap_{n\ge1}\cor_{K_n/K}\bigl(S_p^\pm(E/K_n)\bigr)\subset S_p^\pm(E/K) .
\]
The following theorem is the counterpart for plus/minus Selmer groups of \cite[Theorem 6.1]{Ber2}.

\begin{theorem} \label{perfect-thm}
There is a perfect pairing  
\[
\langle\!\langle\cdot,\cdot\rangle\!\rangle:S^\pm_p(E/K)\big/US^\pm_p(E/K)\times \Sel^\pm_{p^\infty}(E/K_\infty)_{G_\infty}\longrightarrow G_\infty\otimes_{\Zp}\Qp/\Zp.
\]
\end{theorem}

\begin{proof}
Using Proposition \ref{prop:tatepairing}, this follows as in the proof of \cite[Theorem~6.1]{Ber2}, where we replace the control theorem used to prove \cite[Lemma~6.4]{Ber2} with the plus/minus analogue provided by Proposition~\ref{prop:control}.
\end{proof}

This in turn gives the following

\begin{corollary}\label{cor:key}
The $\Lambda$-module $\Sel_{p^\infty}^\pm(E/K_\infty)$ admits no proper $\Lambda$-submodule of finite index if and only if $S_p^\pm(E/K)\big/US_p^\pm(E/K)$ has no $\Zp$-torsion.
\end{corollary}

\begin{proof}
Using Theorem \ref{perfect-thm} in place of \cite[Theorem 6.1]{Ber2}, one can proceed exactly as in the proof of \cite[Corollary~6.2]{Ber2}, which is completely algebraic in nature and does not depend on the definition of the Selmer groups involved.
\end{proof}

\subsection{Computation of universal norms}\label{sec:computation-of-universal-norms}

Under appropriate hypotheses, we now compute the universal norm submodule $US_p^\pm(E/K)$, thus extending \cite[Theorem~7.1]{Ber2} to our setting. In doing so, we follow \cite[\S7]{Ber2} closely. We begin with a generalization of \cite[Lemma~9]{bertolini-lincei}. Recall that $R_n=R_{1,n}=(\ZZ/p\ZZ)[G_n]$ and $\fE^\pm_n=\fE^\pm_{1,n}$.

\begin{lemma} \label{lem:free}
If $\fE_n^\pm\ne0$, then $\Sel_p^\pm(E/K_n)$ admits  a free $R_n$-submodule $U_n^\pm$ such that $\fE_n^\pm\subset U_n^\pm$. 
\end{lemma}

\begin{proof}
For $m\in\{n,n+1\}$, fix a generator $\gamma_{m}$ of $G_m$. Recall from Corollary~\ref{cor:cyclic} that $\fE_m^\pm$ is a cyclic  $R_m$-module. Consequently, \cite[Lemma~3]{bertolini-lincei} says that 
\begin{equation}
\fE_m^\pm\simeq R_m/(\gamma_m-1)^{p^m-t_m^\pm}\simeq (\gamma_m-1)^{t_m^\pm}R_m,
    \label{eq:cyclic}
\end{equation}
where $t_m^\pm=p^m-\dim_{\Fp}(\fE_m^\pm)$. Note that $\cG_{n+1/n}=\langle \gamma_{n+1}^{p^n}\rangle$, which is cyclic of order $p$. Then
\begin{align*}
\bigl((\gamma_{n+1}-1)^{s}R_{n+1}\bigr)^\cG&=\Big(1+\gamma_{n+1}^{p^n}+\cdots +\gamma_{n+1}^{p^n(p-1)}\Big)(\gamma_{n+1}-1)^{s}R_{n+1}\\
&=(\gamma_{n+1}-1)^{p^{n+1}-p^n+s}R_{n+1}
\end{align*}
for all $0\le s\le p^{n+1}$. In particular, isomorphism \eqref{eq:cyclic} tells us that
\[
\dim_{\Fp}\bigl((\gamma_{n+1}-1)^{s}R_{n+1}\bigr)^\cG=p^n-s.
\]
Recall from Lemma \ref{lem:embeds} that
\[
\bigl(\fE_{n+1}^\pm\bigr)^{\cG_{n+1/n}}=\fE_n^\pm.
\]
Since $\dim_{\Fp}\fE_{n+1}^\pm=p^{n+1}-t^\pm_{n+1}$, we have 
\[
\dim_{\Fp}\bigl(\fE_{n+1}^\pm\bigr)^{\cG}=p^n-t^\pm_{n+1}=p^n-t^\pm_n,
\]
so $t^\pm_{n+1}=t^\pm_n$. Let us write $t^\pm$ for this common value and define the cyclic $R_{n+1}$-module
\[
U_n^\pm:=(\gamma_{n+1}-1)^{p^{n+1}-p^n-t^\pm}\fE_{n+1}^\pm\simeq(\gamma_{n+1}-1)^{p^{n+1}-p^n}R_{n+1}.
\]
Note that $U_n^\pm$ is contained in $\fE_{n+1}^\pm$ by definition and is isomorphic to $R_{n+1}^{\cG_{n+1/n}}$. Therefore $U_n^\pm$ is invariant under $\cG_{n+1/n}$, which makes it an $R_n$-module. On the one hand, $U_n^\pm$ contains $(\fE_{n+1}^\pm)^{\cG_{n+1/n}}=\fE_n^\pm$; on the other hand, $U^\pm_n$ is contained in $\Sel_p^\pm(E/K_{n+1})^{\cG_{n+1/n}}\simeq \Sel_p^\pm(E/K_n)$. Finally, we  deduce from \eqref{eq:cyclic} that $\dim_{\Fp}U_n^\pm=p^n$, hence $U_n^\pm$ is free of rank one over $R_n$.
\end{proof}

Now we introduce plus and minus analogues of Shafarevich--Tate groups over finite layers of $K_\infty/K$ and, following \cite[\S7]{Ber2}, relative versions of them as well. 

\begin{definition}
The \emph{plus and minus $p^m$-Shafarevich--Tate groups} of $E$ over $K_n$ are
\[
\sha_{p^m}^\pm(E/K_n):=\Sel_{p^m}^\pm(E/K_n)\big/\HH^\pm_n[p^m].
\]
The \emph{relative plus and minus $p^m$-Shafarevich--Tate groups} are
\[
\sha_{p^m}^\pm(E,K_{n+1}/K_n):=\ker\Big(\sha_{p^m}^\pm(E/K_n)\longrightarrow \sha_{p^m}^\pm(E/K_{n+1})\Big),
\]
where the map on the right is induced by restriction.
\end{definition}

Let $U_n^\pm$ be the free $R_n$-submodule of $\Sel_p^\pm(E/K_n)$ from Lemma \ref{lem:free}, and recall the labels that were assigned to sets of hypotheses in Section \ref{sec:hypotheses}.

\begin{lemma}\label{lem:Un}
Suppose that $\sha^\pm_p(E,K_{n+1}/K_n)=0$ and $\fE^\pm_n\ne0$. Then $U_n^\pm\subset\HH_n^\pm[p]$.
\end{lemma}

\begin{proof}
Consider the commutative diagram
\[
\xymatrix@C=20pt{
0\ar[r]&\HH_{n}^\pm[p]\ar[r]\ar[d]&\Sel_p^\pm(E/K_n)\ar[r]\ar[d]^{\simeq}&\sha^\pm_p(E/K_n)\ar[r]\ar[d]&0\\
0\ar[r]&\left(\HH_{n+1}^\pm[p]\right)^{\cG_{n+1/n}}\ar[r]&\Sel_p^\pm(E/K_{n+1})^{\cG_{n+1/n}}\ar[r]&\sha_p^\pm(E/K_{n+1})^{\cG_{n+1/n}}\ar[r]&0
}
\]
in which the middle vertical isomorphism comes from Proposition~\ref{prop:control}. By applying the snake lemma, one checks that
\begin{equation} \label{H-eq}
\HH_n^\pm[p]=\left(\HH_{n+1}^\pm[p]\right)^{\cG_{n+1/n}}.
\end{equation}
By construction, $U_n^\pm$ is contained in the module on the right-hand side of \eqref{H-eq}, and the result follows.
\end{proof}

\begin{lemma} \label{lem:noSha}
There is an equality 
\[
S_p^\pm (E/K)=\varprojlim_m \HH_0^\pm[p^m].
\]
\end{lemma}

\begin{proof}
By Lemma~\ref{lem:same}, $\HH_0^\pm[p^m]=E(K)/p^mE(K)$ and $\Sel_{p^m}^\pm(E/K)=\Sel_{p^m}(E/K)$, therefore there is a short exact sequence
\begin{equation} \label{sel-sha-eq}
0\longrightarrow\HH_0^\pm[p^m]\longrightarrow\Sel_{p^m}^\pm(E/K)\longrightarrow\sha_{p^m}(E/K)\longrightarrow0.
\end{equation}
But $\sha_{p^m}(E/K)$ is finite and bounded independently of $m$ (\cite[Theorem A]{Kol-Euler}), hence the result follows upon taking inverse limits in \eqref{sel-sha-eq}.
\end{proof}

Recall from $\S$\ref{sec:hypotheses} the following hypothesis.
\begin{itemize}
\item[\textbf{(mod $p$)}] For an elliptic curve $E/\Q$ satisfying both $a_p(E)=0$ and $\textbf{(Norm})$, we have  $\sha^\pm_p(E,K_{n+1}/K_n)=0$ and $\fE^\pm_n\ne0$ for all $n \in \N$.  
\end{itemize}

\begin{remark} 
The condition $\fE^\pm_n\ne0$ has been established in some cases which overlap with our own; see \cite[Theorem 4.6]{burungale}.
\end{remark}

The following theorem is the main result of this article.

\begin{theorem} \label{thm:main-no-finite-index}
Suppose that hypothesis \textbf{(mod $p$)} is satisfied. Then $US_p^\pm(E/K)$ is free of rank one over $\Zp$ and the $\Lambda$-module $\Sel_{p^\infty}^\pm(E/K_\infty)$ admits no proper $\Lambda$-submodule of finite index.
\end{theorem}

\begin{proof}
We first observe that $E(K)$ is $p$-torsion free; this is a consequence of our hypotheses that $E$ has good supersingular reduction at primes above $p$ and that $p$ is unramified in $K$. Thus it follows from the proof of Lemma~\ref{lem:noSha} that $S_p^\pm(E/K)\simeq E(K)\otimes_\Z\Z_p$ is a free $\Zp$-module of finite rank. Furthermore, by \cite[Theorem~1.4]{LV-BUMI} (see also \cite[Theorem~A]{CastellaWan}), the $\Lambda$-module $\Sel_{p^\infty}^\pm(E/K_\infty)^\vee$ has rank one. Combining this with Proposition~\ref{prop:PR}, and arguing as in \cite[\S 3.2]{Ber1}, one sees that
\[ \rank_{\Z_p}US_p^\pm(E/K)=\rank_\Lambda\hat S_p^\pm(E/K_\infty)=1. \]
Then it suffices to show that $US_p^\pm (E/K)$ contains a non-trivial element of $S_p^\pm(E/K)$ not divisible by $p$. This would imply that $US_p^\pm (E/K)\simeq \Zp$ and that the $\Zp$-module  $S^\pm_p(E/K)\big/US^\pm_p(E/K)$ is torsion-free. The last statement of the theorem would then follow from Corollary~\ref{cor:key}.

Thus let us set $T\HH_n^\pm:=\varprojlim_m \HH_n^\pm[p^m]$ for the $p$-adic Tate module of $\HH_n^\pm$; then
\[
T\HH_n^\pm\big/pT\HH_n^\pm=\HH_n^\pm[p].
\]
Lemma~\ref{lem:Un} says that the free $R_n$-module $U_n^\pm$ from Lemma~\ref{lem:free} lies inside $\HH_n^\pm[p]$. Let $\tilde U_n^\pm$ be a free $\Zp[G_n]$-submodule of $T\HH_n^\pm$ of rank one lifting $U_n^\pm$ modulo $p$, generated by an element $v_n$. Then $\cor_{K_n/K}(v_n)$ is not divisible by $p$, thanks to the freeness of $\tilde U_n^\pm$. Lemma~\ref{lem:noSha} tells us that
\[
\cor_{K_n/K}(v_n)\in T\HH_0^\pm=US_p^\pm(E/K).
\]
By compactness, we may find a subsequence $\bigl(\cor_{K_{n_i}/K}(v_{n_i})\bigr)_{i\ge 0}$ converging to an element of  $S_p^\pm(E/K)$ that lies in $US_p^\pm(E/K)$ and is not divisible by $p$, as required.
\end{proof}

\section{An application: variation of Iwasawa invariants} \label{sec:application-variation}

\subsection{$p$-congruent elliptic curves}

In this section, we illustrate one application of Theorem \ref{thm:main-no-finite-index}. Recall that two elliptic curves $E_1/\Q$ and $E_2/\Q$ are \emph{$p$-congruent} (or \emph{congruent modulo $p$}) if $E[p]\simeq E'[p]$ as $G_\Q$-modules. Accordingly, in $\S$\ref{sec:hypotheses}, we gave a name to the following hypothesis.
\begin{itemize} 

\item[\textbf{(Cong)}] Given two elliptic curves $E_1/\Q$ and $E_2/\Q$, both of which satisfy \textbf{(mod $p$)}, there is an isomorphism $E_1[p] \simeq E_2[p]$ of $G_\Q$-modules.  

\end{itemize} 

In \cite{gv}, Greenberg and Vatsal showed that the Iwasawa invariants (defined over the cyclotomic $\Z_p$-extension of $\Q$) of $p$-congruent, $p$-ordinary elliptic curves are related by an explicit formula, and they used this to prove many cases of the Iwasawa Main Conjecture. These results were later extended to Hida families by Emerton, Pollack and Weston (\cite{epw}), while results in the non-ordinary case were established for elliptic curves by B. D. Kim (\cite{kim09}) and for more general modular forms by the first two named authors (\cite{hat-lei1}).

An analogue for the anticyclotomic $\Z_p$-extension $K_\infty$ of an imaginary quadratic field $K$ of the results in \cite{epw} was obtained by Castella, C.-H. Kim and Longo (\cite{CasKimLong}) under the assumption that the tame level $N$ of the Hida family is divisible by an \emph{odd} number of primes that are inert in $K$; in this setting, results on the vanishing of the $\mu$-invariant were obtained by Pollack and Weston (\cite{PolWes}). When $N$ satisfies this divisibility hypothesis, the usual Selmer groups over $K_\infty$ are expected to be $\Lambda$-cotorsion, and the arguments regarding the variation of Iwasawa invariants closely mirror the cyclotomic case.

As we have seen in previous sections, when all the primes dividing $N$ split in $K$ the classical Selmer groups are not $\Lambda$-cotorsion, and extra care must be taken (this happens, more generally, when $N$ is divisible by an \emph{even} number of inert primes). Recently, analogues in this setting of the Greenberg--Vatsal result was proved for $p$-ordinary modular forms by the first two named authors (\cite{hat-lei2, MRL2}). In this final section of the paper, we will extend these results  to setting of $p$-supersingular elliptic curves.

%Throughout this section, let $E$ and $E'$ be elliptic curves defined over $\Q$ of conductor $N$ and $N'$, respectively, with $a_p(E)=a_p(E')=0$. We retain the notation and hypotheses from the rest of the paper, in particular \S \ref{section:notation}, where now we assume that the imaginary quadratic field $K$ satisfies the Heegner hypothesis with respect to both $N$ and $N'$.

\subsection{Definitions of Iwasawa invariants} \label{iwasawa-subsec}

Let us briefly recall the definition of the Iwasawa invariants. As explained in \S \ref{iwasawa-algebra-subsec}, we identify $\Lambda$ with $\Z_p[\![X]\!]$ via our fixed topological generator $\gamma_\infty$ of $G_\infty$. Let $M$ be a finitely generated $\Lambda$-module; there is a pseudo-isomorphism, i.e., a map with finite kernel and cokernel
\begin{equation}\label{eq:pseudo}
M\sim \Lambda^{\oplus r}\oplus \bigoplus_{i=1}^s \Lambda/(p^{a_i})\oplus \bigoplus_{j=1}^t\Lambda\big/\bigl(F_j^{n_j}\bigr)
\end{equation}
for suitable integers $r,s,t\ge 0$, $a_i,n_j\ge1$ and irreducible Weierstrass polynomials $F_j\in\Z_p[X]$ (see, e.g., \cite[Theorem 5.3.8]{NSW}). The \emph{$\mu$-invariant} and the \emph{$\lambda$-invariant} of $M$ are
\[ \mu(M):=\sum_{i=1}^s a_i,\qquad \lambda(M):=\sum_{j=1}^tn_j\deg(F_j). \]

Continue to assume that $M$ is a finitely-generated $\Lambda$-module $M$, and let $M_\tor$ denote the maximal torsion submodule of $M$. Recall from \cite[\S1.8 and \S3.1]{Jan} that there is an exact sequence of $\Lambda$-modules
\begin{equation}
   0\rightarrow M_\tor\rightarrow M\rightarrow M^{++}\rightarrow T_2(M)\rightarrow 0,
\label{eq:jan}
\end{equation}
where $M^{++}$ is the reflexive hull of $M$, which is free over $\Lambda$ and $T_2(M)$ is finite. In the same way as \cite[Definition 3.3]{MRL2}, we define an integer $c(M)$ as follows.

\begin{definition}\label{def:error-term}
  Let $M$ be a finitely generated module over $\Lambda$. We define  $ c(M)$ to be the unique integer satisfying the equation
  \[
 |T_2(M)[p]|=|\F_p|^{c(M)}.
  \]
\end{definition}

When $M=\Sel^\pm_{p^\infty}(E/K_\infty)^\vee$ is the Pontryagin dual of the plus/minus Selmer group associated to an elliptic curve $E$, we write $c_\pm(E)=c(M)$.

\subsection{Variation of Iwasawa invariants} \label{iwasawa-variation-subsec}

For the rest of this paper, we fix elliptic curves $E_1$ and $E_2$ with square-free conductors $N_1$ and $N_2$, respectively, such that $E_1[p] \simeq E_2[p]$ as $G_\Q$-modules and $a_p(E_i)=0$ for $i=1,2$. Moreover, we assume that, taken together, the field $K$, the prime $p$, and the product of the conductors $N_1N_2$  satisfy \textbf{(Heeg)}. We also assume that both curves satisfy the hypothesis  \textbf{(Tam)}.%We take $S$ to be the finite set of places of $K_\infty$ dividing $N_1N_2p\infty$ and we fix a subset $\Sigma \subset S$ that does not contain the archimedean primes or the primes above $p$ and that does include all the primes above which either $E_1[p]$ or $E_2[p]$ is ramified.

Our goal is to relate the Iwasawa invariants of $\Sel^\pm(E_i/K_{\infty})$ for $i=1,2$ (or rather, the Pontryagin duals of these Selmer groups). To ease notation, let us write $\mathcal{X}^\pm(E_i)=\Sel^\pm_{p^\infty}(E_i/K_\infty)^\vee$ for the Pontryagin dual of each Selmer group, and let us set
\[
\mu_\pm(E_i):=\mu\bigl(\mathcal{X}^\pm(E_i) \bigr) \quad \text{and} \quad \lambda_\pm(E_i):=\lambda\bigl(\mathcal{X}^\pm(E_i)).
\]
Our first step is to observe that the $G_\Q$-isomorphism $E_1[p] \simeq E_2[p]$ induces an isomorphism of Selmer groups $$\Sel^\pm_p(E_1/K_\infty) \simeq \Sel^\pm_p(E_2/K_\infty).$$ This amounts to showing that we have isomorphisms between the local conditions defining the Selmer groups. For the places away from $p$, the local condition is the unramified condition, so this is clear, and it remains only to check the local condition at places above $p$. But we have assumed that $p$ splits in $K/\Q$, hence for a prime $\mathfrak{p} \mid p$ of $K$ above $p$, we have $K_{\mathfrak{p}} = \Q_p$, in which case this compatibility is well-established; see e.g. \cite[Proposition 2.8]{kim09} and the proof of Theorem 4.1.1 of \cite{AhmedLim}.

Thus, in light of Proposition \ref{prop:control}, we have an isomorphism 
\[
\Sel^\pm_{p^\infty}(E_1 / K_\infty)[p] \simeq \Sel^\pm_{p^\infty}(E_2 / K_\infty)[p].
\]
By duality, we have established the following result.

\begin{proposition}\label{prop:congruence-gives-isomorphic-selmer}
  Suppose $E_1/\Q$ and $E_2/\Q$ are elliptic curves satisfying $a_p(E_i)=0$ and  \textbf{(Tam)} for $i=1,2$ and that $E_1[p] \simeq E_2[p]$ as $G_\Q$-modules. Assume further that the hypothesis $\textbf{(Heeg)}$ holds. Then we have an isomorphism
  \[
\mathcal{X}^\pm(E_1) / p \simeq \mathcal{X}^\pm(E_2) / p.
  \]
\end{proposition}

Recall that, when $E$ is supersingular, $\mu_\pm(E)$ is always expected to vanish, and it is known to vanish in some cases (see, e.g., \cite[Theorem B]{Matar19}). Furthermore, when these $\mu$-invariants vanish, the corresponding $\lambda$-invariants can often be related to each other in an explicit manner, in the spirit of Greenerg-Vatsal \cite{gv}. In our present setting, we have the following result.

\begin{theorem}\label{thm:main2} \label{thm:variation-of-iwasawa-invariants}
Let $E_1/\Q$ and $E_2/\Q$ be elliptic curves for which the hypotheses $\textbf{(Heeg)}$ and $\textbf{(Cong)}$ hold. Then
\[
\mu_\pm(E_1)=0\; \Longleftrightarrow\; \mu_\pm(E_2)=0.
\]
Suppose that $\mu_\pm(E_1)=\mu_\pm(E_2)=0$. Then
\[
\lambda_\pm(E_1) + c_\pm(E_1) = \lambda_\pm(E_2) + c_\pm(E_2).
\]
\end{theorem}

\begin{proof}
Our assumptions that $E_1[p] \simeq E_2[p]$ and that $\textbf{(Heeg)}$ holds allow us to invoke Proposition \ref{prop:congruence-gives-isomorphic-selmer} to deduce an isomorphism
\[
\mathcal{X}^\pm(E_1) / p \simeq \mathcal{X}^\pm(E_2) / p.
\]
By \cite[Theorem~1.4]{LV-BUMI}, both $\mathcal{X}_\pm(E_1)$ and $\mathcal{X}_\pm(E_2)$ have $\Lambda$-rank $1$, so we may apply \cite[Corollary 2.4]{hat-lei2} to conclude that
\[
\mu_\pm(E_1) =0 \;\Longleftrightarrow\; \mu_\pm(E_2)=0.
\]
Assume that the $\mu_\pm$-invariants vanish. Then since \textbf{(Cong)} holds, we know from Theorem \ref{thm:main-no-finite-index} (and duality) that each $\mathcal{X}^\pm(E_i)$ contains no non-trivial finite $\Lambda$-submodules, so we may apply \cite[Corollary 3.4]{MRL2} to deduce the claimed relationship between the $\lambda_\pm$-invariants, and we are finished.
\end{proof}
\bibliographystyle{amsplain}
\bibliography{references}
\end{document}